\documentclass{article}
\usepackage[left=4cm, right=4cm, top=2.54cm, bottom=2.54cm]{geometry}
\usepackage{amsmath, amssymb, amsthm}
\usepackage[english]{babel} 
\usepackage{mathrsfs} 
\providecommand{\abs}[1]{\lvert#1\rvert}

\providecommand{\absol}[1]{\big{\lvert}#1\big{\rvert}}

 \theoremstyle{definition}
 \newtheorem{theorem}{Theorem}[section]
 
 \newtheorem{corollary}[theorem]{Corollary}
 \newtheorem{lemma}[theorem]{Lemma}
 \newtheorem{proposition}[theorem]{Proposition}
 
 \theoremstyle{definition}
 \newtheorem{defi}[theorem]{Definition}

 \theoremstyle{definition}
 \newtheorem{remark}[theorem]{Remark}

 \numberwithin{equation}{section}

\begin{document}
 
\title{Reconstruction of solutions to a generalized Moisil-Teodorescu system in Jordan domains with rectifiable boundary}

\author{Daniel Gonz\'alez-Campos$^{(1)}$, Marco Antonio P\'erez-de la Rosa$^{(2)}$\\and\\ Juan Bory-Reyes$^{(3)}$}

\date{ \small  $^{(1)}$   Escuela Superior de F\'isica y Matem\'aticas.  Instituto Polit\'ecnico Nacional. CDMX. 07738. M\'exico. \\ E-mail: daniel\_uz13@hotmail.com \\
	$^{(2)}$ Department of Actuarial Sciences, Physics and Mathematics, Universidad de las Am\'ericas Puebla.
	San Andr\'es Cholula, Puebla. 72810. M\'exico. \\ Email: marco.perez@udlap.mx \\
	$^{(3)}$ ESIME-Zacatenco. Instituto Polit\'ecnico Nacional. CDMX. 07738. M\'exico. \\ E-mail: juanboryreyes@yahoo.com }

\maketitle	
\begin{abstract}
	
In this paper we consider the problem of reconstructing solutions to a generalized Moisil-Teodorescu system in Jordan domains of $\mathbb{R}^{3}$ with rectifiable boundary. In order to determine conditions for existence of solutions to the problem we embed the system in an appropriate generalized quaternionic setting.
\vspace{0.3cm}

\small{
} 
\end{abstract}

\section{Introduction}

Quaternionic analysis is a function theory that offers a natural and elegant extension of classic complex analysis in the plane to the quaternion skew-field (denoted by $\mathbb{H}$) generated by a real unit $1$ and the imaginary units $\textbf{i}$, $\textbf{j}$, $\textbf{k}$. This analysis relies heavily on results on quaternion-valued functions that are defined in domains of $\mathbb{R}^{s}, s=3, 4$ and that are null solutions of generalized Cauchy-Riemann or Dirac operator \cite{23, 24}.

Nowadays, quaternionic analysis has proven to be a good tool to study numerous mathematical models of spatial physical phenomena related to many different extensions of the Moisil-Teodorescu system, see \cite{9, 10, 30, 31, 32, 33, 34}.

The use of a general orthonormal basis $\psi:=\{\psi^1,\psi^2,\psi^3\} \in\mathbb{H}^3$, called structural set, instead of the standard basis in $\mathbb{R}^{3}$, and introducing a generalized Moisil-Teodorescu system associated to $\psi$ are the cornerstone of a generalized quaternionic analysis. 

The pioneers works on this theory, centered around the concept of $\psi$-hyperholomorphic functions defined on domains in $\mathbb{R}^{3}$ (or $\mathbb{R}^{4}$), were proposed independently by Naser \cite{3}, Nono \cite{4,5}, Shapiro and Vasilevsky \cite{6,7}. More information about the importance of the structural sets can be found in \cite{14, 21, 22, 25, 26, 27, 28}.

The structural set $\psi^\theta:=\{\textbf{i},\,  \textbf{i}e^{\textbf{i}\theta}\textbf{j},\, e^{\textbf{i}\theta}\textbf{j}\}$, for $0\leqslant\theta\leqslant2\pi$ fixed and its associated operator 
$${^{\psi^{\theta}}}D:=\displaystyle\frac{\partial}{\partial x_{1}}\textbf{i}+\frac{\partial}{\partial x_{2}}\textbf{i}e^{\textbf{i}\theta}\textbf{j}+\frac{\partial}{\partial x_{3}} e^{\textbf{i}\theta}\textbf{j}$$ 
are used in \cite{1} to study solvability conditions for a non-homogeneous generalized Moisil-Teodorescu system. To achieve our goals, we consider the homogeneus case of such Moisil-Teodorescu system.
	
Let $\Omega\subset \mathbb{R}^{3}$ be a Jordan domain (see \cite{8}) with rectifiable boundary $\Gamma$ (see \cite{2}), i.e. $\Gamma$ is the image of some bounded subset of $\mathbb{R}^2$ under a Lipschitz mapping, and let $0\leq\theta\leq 2\pi$:
	\begin{equation} \label{sed}
	\left\{
	\begin{array}{rcl}
	-\displaystyle \frac{\partial f_{1}}{\partial x_{1}}+\left(\frac{\partial f_{2}}{\partial x_{2}}-\frac{\partial f_{3}}{\partial x_{3}}\right)\sin\theta-\left(\frac{\partial f_{3}}{\partial x_{2}}+\frac{\partial f_{2}}{\partial x_{3}}\right)\cos\theta & = & 0, 
	\\ {}\\	\displaystyle {\left(\frac{\partial f_{3}}{\partial x_{3}}-\frac{\partial f_{2}}{\partial x_{2}}\right)}\cos\theta-\left(\frac{\partial f_{3}}{\partial x_{2}}+\frac{\partial f_{2}}{\partial x_{3}}\right)\sin\theta & = & 0,
	\\ {}\\ \displaystyle {-\frac{\partial f_{3}}{\partial x_{1}}+\frac{\partial f_{1}}{\partial x_{3}}\sin\theta+\frac{\partial f_{1}}{\partial x_{2}}\cos\theta}  & = & 0, \\ {}\\
	\displaystyle {\frac{\partial f_{2}}{\partial x_{1}}-\frac{\partial f_{1}}{\partial x_{3}}\cos\theta+\frac{\partial f_{1}}{\partial x_{2}}\sin\theta} & = & 0, 
	\end{array}
	\right. 
	\end{equation}
with unknown functions $f_m, m=1,2,3$, which belong to $C^1(\Omega,\mathbb{C})\cap C^0(\Omega\cup\Gamma,\mathbb{C})$.

Throughout this work, rectifiable surfaces are used as they are essentially the largest class where many basic properties of smooth surfaces have reasonable analogues. For example, Ra\-de\-ma\-cher's theorem (see \cite{19}) ensures that for a rectifiable surface $\Gamma$ there exists conventional plane tangent for almost every point of $\Gamma$, and exists therefore an outward pointing normal vector $\vec{\nu}(\xi)=(\nu_{1}(\xi),\,\nu_{2}(\xi),\,\nu_{3}(\xi))$ to $\Gamma$ for almost every $\xi\in\Gamma$.

It is worth pointing out in advance that system $(\ref{sed})$ corresponds to the equation ${^{\psi^{\theta}}}D[\mathbf{f}]\equiv 0$ in $\Omega$ for $\mathbf{f}:=(f_{1}, f_{2}, f_{3}): \Omega \rightarrow \mathbb{C}^{3}$. A detailed exposition of notations and definitions will be given in Section 2. 
	
We now indicate some relevant particular cases of system (\ref{sed}).	
	\begin{itemize}
		\item     \textbf{Div-rot system} (see \cite{10, 11}).
	\end{itemize}
	\begin{equation} \label{sisdr}
	\left\{
	\begin{array}{rcl}
	-\displaystyle \frac{\partial f_{1}}{\partial x_{1}}-\frac{\partial f_{2}}{\partial x_{2}}-\frac{\partial f_{3}}{\partial x_{3}}  =  0, \quad \quad \quad \quad \quad \quad \quad \quad \quad \quad \quad \quad \quad \quad
	\\ {}\\	\displaystyle \frac{\partial f_{2}}{ \partial x_{3}}-\frac{\partial f_{3}}{\partial x_{2}}  =  0, \quad	-\frac{\partial f_{2}}{\partial x_{1}}+\frac{\partial f_{1}}{\partial x_{2}}   =  0, \quad
	\frac{\partial f_{3}}{\partial x_{1}}-\frac{\partial f_{1}}{\partial x_{3}}  =  0. 
	\end{array}
	\right. 
	\end{equation}
	
The system \eqref{sisdr} is better known in the following form:
	\begin{equation} \label{div}
	\left\{
	\begin{array}{rcl}
	\text{div} \vec{f} & = & 0, 
	\\ {}\\	\text{rot} \vec{f} & = & 0,
	\end{array}
	\right.
	\end{equation}
where $\vec{f}=f_{1}\textbf{i}+f_{2}\textbf{j}+f_{3}\textbf{k}$. Whenever $f_1\mathbf{i}+f_2\mathbf{j}+f_3\mathbf{k}$ is a solution of (\ref{sed}), then $f_1\mathbf{i}+f_3\mathbf{j}+f_2\mathbf{k}$ is so of (\ref{div}) when $\theta=0$ is taken.
	\begin{itemize}
		\item  \textbf{The homogeneous Cimmino system}  (see \cite{12}).
	\end{itemize}
A particular case of the homogeneous Cimmino system is given by:
	\begin{equation}
	\left\{
	\begin{array}{rcl}
	-\displaystyle \frac{\partial f_{1}}{\partial x_{1}}+\frac{\partial f_{2}}{\partial x_{2}}-\frac{\partial f_{3}}{\partial x_{3}}  =  0, \quad \quad \quad \quad \quad \quad \quad \quad \quad \quad \quad \quad \quad \quad 
	\\ {}\\	-\displaystyle \frac{\partial f_{3}}{\partial x_{2}}-\frac{\partial f_{2}}{\partial x_{3}} =  0, \quad
	{-\frac{\partial f_{3}}{\partial x_{1}}+\frac{\partial f_{1}}{\partial x_{3}} } = 0, \quad
	{\frac{\partial f_{2}}{\partial x_{1}}+\frac{\partial f_{1}}{\partial x_{2}}}  =  0. 
	\end{array}
	\right. 
	\end{equation}
where $\vec{f}=f_{1}\textbf{i}+f_{2}\textbf{j}+f_{3}\textbf{k}$ and each function $f_{m}$, depends only of $(x_{1}, x_{2}, x_{3})$. This case is obtained from \eqref{sed} for $\displaystyle \theta=\frac{\pi}{2}$. 
	\begin{itemize}
		\item \textbf{The Riesz system} (see \cite{13, 14, 15}).
	\end{itemize}	
Let $f:(x_{0}, x_{1}, x_{2})\in \mathbb{R}^{3} \to {\text{span}_{\mathbb{R}^{3}}} \{1, \textbf{i}, \textbf{j}\}$ 
	\begin{equation}
	\left\{
	\begin{array}{rcl}
	\displaystyle \frac{\partial f_{0}}{\partial x_{0}}-\frac{\partial f_{1}}{\partial x_{1}}-\frac{\partial f_{2}}{\partial x_{2}}  =  0, 
	\quad \quad \quad \quad \quad \quad \quad \quad \quad \quad \quad \quad \quad \,
	\\ {}\\	\displaystyle \frac{\partial f_{0}}{\partial x_{1}}+\frac{\partial f_{1}}{\partial x_{0}}  =  0,
	\quad	\frac{\partial f_{0}}{\partial x_{2}}+\frac{\partial f_{2}}{\partial x_{0}}   =  0, \quad
	\frac{\partial f_{1}}{\partial x_{2}}-\frac{\partial f_{2}}{\partial x_{1}} =  0, \,\,\,
	\end{array}
	\right.
	\end{equation}
which is equivalent with the so-called Riesz system
	\begin{equation}\label{riesz}
	\left\{
	\begin{array}{rcl}
	\text{div} \bar{f} & = & 0, 
	\\ {}\\	\text{rot} \bar{f} & = & 0,
	\end{array}
	\right. 
	\end{equation}
where $\bar{f}:=f_{0}-f_{1}\textbf{i}-f_{2}\textbf{j}$. If $f_1\mathbf{i}+f_2\mathbf{j}+f_3\mathbf{k}$ is a solution of (\ref{sed}), then  $f_1+f_3\mathbf{i}+f_2\mathbf{j}$ is so of the inhomogeneous system (\ref{riesz}) for $\theta=\pi$.
	\begin{itemize}
		\item \textbf{Another particular case.}
	\end{itemize}
	\begin{equation} \label{casop}
	\left\{
	\begin{array}{rcl}
	-\displaystyle \frac{\partial f_{1}}{\partial x_{1}}-\frac{\partial f_{2}}{\partial x_{2}}+\frac{\partial f_{3}}{\partial x_{3}}  =  0,
	\quad \quad \quad \quad \quad \quad \quad \quad \quad \quad \quad \quad \quad \quad  
	\\ {}\\	\displaystyle \frac{\partial f_{3}}{\partial x_{2}}+\frac{\partial f_{2}}{\partial x_{3}} =  0,
	\quad  	{-\frac{\partial f_{3}}{\partial x_{1}}-\frac{\partial f_{1}}{\partial x_{3}}}   =  0, \quad  
	{\frac{\partial f_{2}}{\partial x_{1}}-\frac{\partial f_{1}}{\partial x_{2}}}  =  0. 
	\end{array}
	\right. 
	\end{equation}
	
To get the system \eqref{casop}, take $\displaystyle \theta=\frac{3\pi}{2}$ in \eqref{sed}. This example is related to the time-harmonic relativistic Dirac bispinor theory  (see \cite{9, 16}).
	
The purpose of this work is to study the problem of reconstruction of solutions to the generalized Moisil-Teodorescu system \eqref{sed}, which is formulated as follows: Given a continuous three-dimensional vector field $\mathbf{f}: \Gamma \rightarrow \mathbb{C}^{3}$, under which conditions can $\mathbf{f}$ be decomposed on $\Gamma$ as a sum:
	\begin{equation} \label{des}
	\mathbf{f}(t)=\mathbf{f}^{+}(t)+\mathbf{f}^{-}(t),  \quad \forall \, t\in\Gamma,
	\end{equation}
where $\mathbf{f}^{\pm}$ are extendable to vector fields $\mathbf{F}^{\pm}$ that satisfy the generalized Moisil-Teodorescu system \eqref{sed} in respectively $\Omega_{+}:=\Omega$ and  $\Omega_{-}:=\mathbb{R}^{3}\setminus\{\Omega\cup\Gamma\}$, with $\mathbf{f}^{-}(\infty)=0$?
	
In \cite{2}, is considered the problem of reconstruction of solutions to the Div-rot system \eqref{sisdr} by using quaternionic analysis tools. Our results extend the achievements of \cite{2} to the aforementioned variety of systems.
	
The integral
	\begin{equation}
	\begin{split}
	{^{\psi}}\mathbf{K}_{\Gamma}[\mathbf{f}](x)&
	:=\int_{\Gamma}{\mathscr{K}_{\psi}(x-\xi)\langle\nu_{\psi}(\xi), \mathbf{f}(\xi)\rangle}dS_{\xi}+\int_{\Gamma}{[ \mathscr{K}_{\psi}(x-\xi), [\nu_{\psi}(\xi), \mathbf{f}(\xi)]]}dS_{\xi},
	\end{split}
	\end{equation}
with $x\in\Omega_{\pm}$, plays the role of the Cauchy transform in the theory of three dimensional continuous vector fields $\mathbf{f}$, where $dS_{\xi}$ is the two-dimensional surface area element on $\Gamma$ and  $\nu_{\psi}(\xi):=\sum_{k=1}^{3}\psi^{k}\nu_{k}(\xi)$ is outward pointing normal vector to $\Gamma$.  The function $\mathscr{K}_{\psi}$ plays a similar role in the $\psi$-hyperholomorphic function theory as the Cauchy kernel does in complex analysis.
	
Similarly, the singular Cauchy transform is defined for $t\in\Gamma$ as
	\begin{equation}
	\begin{split}
	{^{\psi}\pmb{\mathscr{S}}}_{\Gamma}[\mathbf{f}](t):&=\int_{\Gamma}{\mathscr{K}_{\psi}(t-\tau)\langle\nu_{\psi}(\tau), (\mathbf{f}(\tau)-\mathbf{f}(t))\rangle}dS_{\tau}\\ & \quad +\int_{\Gamma}{[ \mathscr{K}_{\psi}(t-\tau), [\nu_{\psi}(\tau), (\mathbf{f}(\tau)-\mathbf{f}(t))]]}dS_{\tau}+\mathbf{f}(t),
	\end{split}
	\end{equation}
where the integral is being understood in the sense of the Cauchy principal value.
	
We shall use the following notation:
	\begin{equation}
	\mathscr{M}_{\psi}:=\bigg\{\mathbf{f}: \int_{\Gamma}{\langle\mathscr{K}_{\psi}(x-\xi) ,[\nu_{\psi}(\xi), \mathbf{f}(\xi)]\rangle}dS_{\xi}=0,  x\not\in\Gamma \bigg\},
	\end{equation}
	\begin{equation}
	\mathscr{M}_{\psi}^{*}:=\bigg\{\mathbf{f}: \int_{\Gamma}{\langle\mathscr{K}_{\psi}(x-\xi) ,[\nu_{\psi}(\xi), \mathbf{f}(\xi)]\rangle}dS_{\xi}=0,  x\in\Gamma \bigg\}.
	\end{equation}
The set $\mathscr{M}_{\psi}^{*}$ can be described in purely physical terms (see \cite{20}).
	
After this brief introduction let us give a description of the structure of the paper. Section 2 presents some preliminaries on the $\psi$-hyperholomorphic function theory. Section 3 is devoted to the study the Cauchy transform ${^{\psi^{\theta}}}\mathbf{K}_{\Gamma}[\mathbf{f}]$. Section 4 contains a pair of generalizations of the results presented in \cite[Theorem 3.3 and 3.4]{2}. In Section 5 our main results are stated and proved.
	
\section{Preliminaries}
\subsection{Basics of $\psi$-hyperholomorphic functions theory}
Let $\mathbb{H}:=\mathbb{H(\mathbb{R})}$ and $\mathbb{H(\mathbb{C})}$ denote the sets of real and complex quaternions respectively. If $a\in\mathbb{H}$ or $a\in\mathbb{H(\mathbb{C})}$, then $a=a_{0}+a_{1}\textbf{i}+a_{2}\textbf{j}+a_{3}\textbf{k}$, where the coefficients $a_{k}\in\mathbb{R}$ if $a\in\mathbb{H}$ and  $a_{k}\in\mathbb{(\mathbb{C})}$ if $a\in\mathbb{H(\mathbb{C})}$. The symbols
$\textbf{i}$, $\textbf{j}$ and $\textbf{k}$ denote different imaginary units, i.e. $\textbf{i}^{2}=\textbf{j}^{2}=\textbf{k}^{2}=-1$ and they satisfy the following multiplication rules $\textbf{i}\textbf{j}=-\textbf{j}\textbf{i}=\textbf{k}$; $\textbf{j}\textbf{k}=-\textbf{k}\textbf{j}=\textbf{i}$; $\textbf{k}\textbf{i}=-\textbf{i}\textbf{k}=\textbf{j}$. The (complex) imaginary unit $i\in\mathbb{C}$ commutes with every cuaternionic unit imaginary. 
	
It is known that $\mathbb{H}$ is a skew-field and $\mathbb{H(\mathbb{C})}$ is an associative, non-conmmutative complex algebra with zero divisors.
	
If $a\in\mathbb{H}$ or $a\in\mathbb{H(\mathbb{C})}$, $a$ can be represented as $a=a_{0}+\vec{a}$, with $\vec{a}=a_{1}\textbf{i}+a_{2}\textbf{j}+a_{3}\textbf{k}$,
	$Sc(a):=a_{0}$ is called the scalar part and
	$Vec(a):=\vec{a}$ is called the vector part of the quaternion $a$. Also, if $a\in\mathbb{H(\mathbb{C})}$, $a$ can be represented as $a=\alpha_{1}+i\alpha_{2}$ with $\alpha_{1},\,\alpha_{2}\in\mathbb{H}$.
	
Let $a,\,b\in\mathbb{H(\mathbb{C})}$, the product between these quaternions can be calculated by the formula:
	\begin{equation} \label{pc2}
	ab=a_{0}b_{0}-\langle\vec{a},\vec{b}\rangle+a_{0}\vec{b}+b_{0}\vec{a}+[\vec{a},\vec{b}],
	\end{equation}
where
	\begin{equation} \label{proint}
	\langle\vec{a},\vec{b}\rangle:=\sum_{k=1}^{3} a_{k}b_{k}, \quad
	[\vec{a},\vec{b}]:= \left|\begin{matrix}
	\textbf{i} & \textbf{j} & \textbf{k}\\
	a_{1} & a_{2} & a_{3}\\
	b_{1} & b_{2} & b_{3}
	\end{matrix}\right|.
	\end{equation}
We define the conjugate of $a=a_{0}+\vec{a}\in\mathbb{H(\mathbb{C})}$ by $\overline{a}:=a_{0}-\vec{a}$.
	
The Euclidean norm of a quaternion $a\in\mathbb{H}$ is the number $\abs{a}$ 
	given by: 
	\begin{equation}\label{normar}
	\abs{a}=\sqrt{a\overline{a}}=\sqrt{\overline{a}a}.
	\end{equation}
We define the quaternionic norm of 
	$a\in\mathbb{H(\mathbb{C})} $ by:
	\begin{equation}
	\abs{a}_{c}:=\sqrt{{{\abs {a_{0}}}_{\mathbb{C}}}^{2}+{{\abs {a_{1}}}_{\mathbb{C}}}^{2}+{{\abs {a_{2}}}_{\mathbb{C}}}^{2}+{{\abs {a_{3}}}_{\mathbb{C}}}^{2}},
	\end{equation}
where ${\abs {a_{k}}}_{\mathbb{C}}$ denotes the complex norm of each component of the quaternion $a$. The norm of a complex quaternion $a=\alpha_{1}+i\alpha_{2}$ with $\alpha_{1}, \alpha_{2} \in \mathbb{H}$
can be rewritten in the form
	\begin{equation} \label{nc2}
	{\abs{a}_{c}}=\sqrt{\abs{\alpha_{1}}^2+\abs{\alpha_{2}}^2}.
	\end{equation}
If $a \in \mathbb{H}$, $b \in \mathbb{H(\mathbb{C})}$, then
	\begin{equation}
	{\abs{ab}}_{c}=\abs{a}{\abs{b}}_{c}.
\end{equation}
If $a\in\mathbb{H(\mathbb{C})}$ is not a zero divisor then $\displaystyle a^{-1}:=\frac{\overline{a}}{a\overline{a}}$ is the inverse of $a$.
\begin{subsection}{Notations} 
		\begin{itemize} 
			\item
A complex quaternionic valued function $f:\Omega\to\mathbb{H(\mathbb{C})}$ will be expressed as
			\begin{equation}
			f:=f_{0}+f_{1}\textbf{i}+f_{2}\textbf{j}+f_{3}\textbf{k}.
			\end{equation}
If $f_{0}\equiv 0$ in $\Omega$, $f$ will be called a vector field (written $f=\mathbf{f})$.
			\item We say that $f$ has properties in $\Omega$ such as continuity and real differentiability of order $p$ whenever all $f_{j}$  have these properties. These spaces are usually denoted by $C^{p}(\Omega,\, \mathbb{H(\mathbb{C})})$ with $p\in(\mathbb{N}\cup\{0\})$.
			\item Throughout this work, $Lip_{\mu}(\Omega,\, \mathbb{H(\mathbb{C})})$, $0<\mu\leq 1$,  denotes the set of H\"older continuous functions defined on $\Omega$ with values on $\mathbb{H(\mathbb{C})}$ and H\"older exponent $\mu$. 
		\end{itemize}
	
As defined in \cite{9}, consider on $C^{1}(\Omega,\, \mathbb{H(\mathbb{C})})$ an operator ${^\psi}D$ by the formula
		\begin{equation} \label{opi}
		{^\psi}D[f]:=\sum_{k=1}^{3}\psi^k\frac{\partial f}{\partial x_{k}},
		\end{equation}
where $\psi:=\{\psi^1,\psi^2,\psi^3 \}$, $\psi^k\in\mathbb{H}$. Denote $\overline{\psi}:=\{\overline{\psi^1},\overline{\psi^2},\overline{\psi^3} \}$. Then, the equality
		\begin{equation} \label{opl}
		{{^\psi}D}{^{\overline{\psi}}D}={^{\overline{\psi}}D}{^\psi}D=\Delta_{3},
		\end{equation}
		is true if and only if
		\begin{equation} \label{dce}
		\psi^j\overline{\psi}^k+\psi^k\overline{\psi}^j=2\delta_{jk},
		\end{equation}
for $j, k\in\{1, 2, 3\}$.
		
Any set of real quaternions $\psi:=\{\psi^1,\psi^2,\psi^3 \}$, $\psi^k\in\mathbb{H}$ with the property \eqref{dce} is called structural set.
		
Similarly (see \cite{9}), on $C^{1}(\Omega,\, \mathbb{H(\mathbb{C})})$ is defined an operator $D^\psi$ 
		\begin{equation}
		D^\psi[f]:=\sum_{k=1}^{3}\frac{\partial f}{\partial x_{k}} \psi^k.
		\end{equation}
		
    \begin{defi} \cite{9}. \label{dfhi} We say that $f\in C^{1}(\Omega, \, \mathbb{H(\mathbb{C})})$ is left (right)-$\psi$-hyperholomorphic in $\Omega$ if ${^\psi}D[f](x)=0$ ($D^\psi[f](x)=0$) for all $x\in\Omega$.
		\end{defi}
It is known that a quaternionic valued function can be left-$\psi$-hyperholorphic but no right-$\psi$-hyperholorphic and vice-versa; or can be left-right-$\psi$-hyperholomorphic. 

The function 
		\begin{equation} \label{kernel}
		\mathscr{K}_{\psi}(x):=\frac{1}{4\pi}\frac{(x)_{\psi}}{\abs{x}^3}, \quad x\in\mathbb{R}^{3}\setminus\{0\},
		\end{equation}
where
		\begin{equation}
		(x)_{\psi}:=\sum_{k=1}^{3}x_{k}\psi^{k},
		\end{equation}
is a left-right-$\psi$-hyperholomorphic fundamental solution of $^{\psi}D$ in $x \in \mathbb{R}^{3} \setminus \{0\}$. Observe that $\abs{(x)_{\psi}}=\abs{x}$ for all $ x \in \mathbb{R}^{3}$.		
\end{subsection}
	
Consider the special case of the structural set $\psi:=\psi^\theta:=\{\textbf{i},\,  \textbf{i}e^{\textbf{i}\theta}\textbf{j},\, e^{\textbf{i}\theta}\textbf{j}\}$, for $0\leqslant\theta\leqslant2\pi$ fixed, then the operator 
	${^{\psi^\theta}}D$ 
takes the form
	\begin{equation}
	{^{\psi^{\theta}}}D:=\frac{\partial}{\partial x_{1}}\textbf{i}+\frac{\partial}{\partial x_{2}}\textbf{i}e^{\textbf{i}\theta}\textbf{j}+\frac{\partial}{\partial x_{3}} e^{\textbf{i}\theta}\textbf{j}. 
	\end{equation}
We define the following partial operators for $f\in C^1(\Omega,\mathbb{H(\mathbb{C})})$:
	\begin{equation}
	{^{\psi^{\theta}}}\text{div}[\vec{f}]:=\frac{\partial f_{1}}{\partial x_{1}}+\left({\frac{\partial f_{2}}{\partial x_{2}}-\frac{\partial f_{3}}{\partial x_{3}}}\right)\textbf{i}e^{\textbf{i}\theta},
	\end{equation}
	\begin{equation}
	{^{\psi^{\theta}}}\text{grad}[f_{0}]:=\frac{\partial f_{0}}{\partial x_{1}}\textbf{i}+\frac{\partial f_{0}}{\partial x_{2}}\textbf{i}e^{\textbf{i}\theta}\textbf{j}+\frac{\partial f_{0}}{\partial x_{3}}e^{\textbf{i}\theta}\textbf{j},
	\end{equation}
	\begin{equation}
	\begin{split}
	{^{\psi^{\theta}}}\text{rot}[\vec{f}]:=\left({-\frac{\partial f_{3}}{\partial x_{2}}-\frac{\partial f_{2}}{\partial x_{3}}}\right)e^{\textbf{i}\theta}+\left({-\frac{\partial f_{1}}{\partial x_{3}}\textbf{i}e^{\textbf{i}\theta}-\frac{\partial f_{3}}{\partial x_{1}}}\right)\textbf{j} +\left({\frac{\partial f_{2}}{\partial x_{1}}-\frac{\partial f_{1}}{\partial x_{2}}\textbf{i}e^{\textbf{i}\theta}}\right)\textbf{k}.
	\end{split}
	\end{equation}
If we apply	${^{\psi^{\theta}}}D$ to $f\in C^1(\Omega, \, \mathbb{H(\mathbb{C})})$ we obtain
	\begin{equation}
	{^{\psi^{\theta}}}D[f]=-{^{\psi^{\theta}}}\text{div}[\vec{f}]+{^{\psi^{\theta}}}\text{grad}[f_{0}]+	{^{\psi^{\theta}}}\text{rot}[\vec{f}],
	\end{equation}
which implies that ${^{\psi^{\theta}}}D[f]=0$ is equivalent to
	\begin{equation} \label{eq1}
	-{^{\psi^{\theta}}}\text{div}[\vec{f}]+{^{\psi^{\theta}}}\text{grad}[f_{0}]+	{^{\psi^{\theta}}}\text{rot}[\vec{f}]=0.
	\end{equation}
For $f_{0}=0$, \eqref{eq1} is reduced to 
	\begin{equation} \label{eq2}
	-{^{\psi^{\theta}}}\text{div}[\vec{f}]+{^{\psi^{\theta}}}\text{rot}[\vec{f}]=0.
	\end{equation}
The equation \eqref{eq2}  is equivalent to the system  \eqref{sed}.
	
On the other hand, if we apply $D^{\psi^{\theta}}$ to $f\in C^{1}(\Omega, \, \mathbb{H(\mathbb{C})})$, we get
	\begin{equation} \label{eq3}
	D^{\psi^{\theta}}[f]=-{^{\overline{\psi^{\theta}}}}\text{div}[\vec{f}]+{^{\psi^{\theta}}}\text{grad}[f_{0}]+{^{\overline{\psi^{\theta}}}}\text{rot}[\vec{f}],
	\end{equation}
where
	\begin{equation}
	{^{\overline{\psi^{\theta}}}}\text{div}[\vec{f}]:=\frac{\partial f_{1}}{\partial x_{1}}+\left({\frac{\partial f_{2}}{\partial x_{2}}-\frac{\partial f_{3}}{\partial x_{3}}}\right)\overline{\textbf{i}e^{\textbf{i}\theta}},
	\end{equation}
	\begin{equation}
	\begin{split}
	{^{\overline{\psi^{\theta}}}}\text{rot}[\vec{f}]:=\left({-\frac{\partial f_{3} }{\partial x_{2}}-\frac{\partial f_{2}}{\partial x_{3}}}\right) \overline{e^{\textbf{i}\theta}}-{\frac{\partial f_{1}}{\partial x_{3}}\overline{\textbf{i}e^{\textbf{i}\theta}\textbf{j}}+\frac{\partial f_{3}}{\partial x_{1}}}\textbf{j} -\frac{\partial f_{2}}{\partial x_{1}}\textbf{k}-\frac{\partial f_{1}}{\partial x_{2}}\overline{\textbf{i}e^{\textbf{i}\theta}\textbf{k}}.
	\end{split}
	\end{equation}
If $f_{0}=0$, \eqref{eq3} is reduced to
	\begin{equation} \label{eq4}
	D^{\psi^{\theta}}[f]=-{^{\overline{\psi^{\theta}}}}\text{div}[\vec{f}]+{^{\overline{\psi^{\theta}}}}\text{rot}[\vec{f}].
	\end{equation}	
Then, the equation
\begin{equation} \label{eq5}
-{^{\overline{\psi^{\theta}}}}\text{div}[\vec{f}]+{^{\overline{\psi^{\theta}}}}\text{rot}[\vec{f}]=0.
\end{equation}	
is equivalent to the system of equations \eqref{sed}.
	\begin{defi} \label{dcvl} A continuously differentiable vector field $\mathbf{f}$ defined in $\Omega$ is said to be a $\psi^\theta$-Laplacian vector field if $\mathbf{f}$ is left-right-$\psi^\theta$-hyperholomorphic in $\Omega$.
	\end{defi}
The following Lemma will be used in the proof of our main result.
	\begin{lemma} \label{two-sided} Let $f=f_{0}+\vec{f}\in C^{1}(\Omega, \, \mathbb{H(\mathbb{C})})$. Then $f$ is left-right-$\psi^\theta$-hyperholomorphic in $\Omega$ if and only if ${^{\psi^{\theta}}}\text{grad}[f_{0}](x)=0$, for all $x\in\Omega$ and $\mathbf{f}:=\vec{f}$ is a $\psi^\theta$-Laplacian vector field in $\Omega$.
		\begin{proof} 
The proof is based on the fact that \eqref{eq2} and \eqref{eq5} are equivalent to \eqref{sed}.
		\end{proof}
	\end{lemma}
\subsection{Whitney's extension theorem}
The following results are generalizations of those presented in \cite[Theorem 4.1, Proposition 4.1]{17}, which is due to the fact that the proofs make no appeal to which structural set is assumed. 
	\begin{theorem} \label{inciso3}
Let  ${f}\in Lip_{\mu}(\Gamma, \mathbb{H(\mathbb{R})})$. Define $f^{w}:=\mathcal{X}\mathcal{E}_{0}(f)$, where $\mathcal{X}$ is the characteristic function in $\Omega\cup\Gamma$ and $\mathcal{E}_{0}(f)$ is the Whitney operator (see \cite{29}) applied to $f$. Then 
		\begin{itemize}
			\item [a)] $f^{w}\in Lip_{\mu}(\Omega\cup\Gamma, \mathbb{H(\mathbb{R})})$.
			\item [b)] $\absol{\displaystyle{\frac{\partial f^{w}}{\partial x_{i}}}(x)}\leq c (dist(x,\Gamma))^{\mu-1}$.
			\item [c)] $\abs{{^{\psi^{\theta}}D}f^{w}(x)}\leq c_{1} (dist(x,\Gamma))^{\mu-1}$,
		\end{itemize}
where $c$, $c_{1}$ are constants.
	\end{theorem}
	\begin{proposition} \label{lemmalp} 
Let $f \in Lip_{\mu}(\Gamma , \mathbb{ H(\mathbb{R}) })$. Then $^{\psi^{\theta}}D[f^{w}] \in L_{p}(\mathbb{R}^{3}, \mathbb{H(\mathbb{R})})$ for $p < \displaystyle \frac{1}{1-\mu}$.
\end{proposition}
\subsection{Integral operators for $\psi^\theta$-hyperholomorphic function theory}
Now we enunciate the most essential integral formulas for the theory of $\psi^\theta$ - hyperholomorphic functions, starting with the quaternionic Stokes' formula (see \cite{9}), which is a consequence of the Stokes' theorem in real analysis.
	
Let $f,g\in C^{1}(\Omega,\, \mathbb{H(\mathbb{C})})\cap C(\Omega\cup\Gamma,\, \mathbb{H(\mathbb{C})})$. Then
	\begin{equation}	
	\int_{\Gamma}{g(\xi){\nu_{\psi^{\theta}}}(\xi) f(\xi)}dS_{\xi}=\int_{\Omega}[D^{\psi^{\theta}}[g](\xi) f(\xi)+g(\xi) {^{\psi^{\theta}}}D[f](\xi)]dm({\xi}).
	\end{equation}
	\begin{defi} \cite{9}. Let $f$ be a continuous function defined on $\Omega\cup\Gamma$. The quaternionic Cauchy and Teodorescu transforms of $f$ are defined by
		\begin{equation}
		{^{\psi^{\theta}}}T[f](x):=\int_{\Omega}{\mathscr{K}_{\psi^{\theta}}(x-\xi) f(\xi) }dm(\xi),  \quad x\in \mathbb{R}^{3},
		\end{equation}
		\begin{equation}
		{^{\psi^{\theta}}}K_{\Gamma}[f](x):=-\int_{\Gamma}{\mathscr{K}_{\psi^{\theta}}(x-\xi)\nu_{\psi^{\theta}}(\xi) f(\xi)}dS_{\xi}, \quad x\in \mathbb{R}^{3}\setminus \Gamma.
		\end{equation}
The singular Cauchy transform is defined by
		\begin{equation}\label{singular}
		{^{\psi^{\theta}}}\mathscr{S}_{\Gamma}[f](t):=2{^{\psi^{\theta}}}\Phi[f](t)+f(t), \quad t\in  \Gamma,
		\end{equation}
where
		\begin{equation}
		{^{\psi^{\theta}}}\Phi[f](t):=\lim_{r \to 0} \int_{\Gamma\setminus\Gamma_{t,r}}{\mathscr{K}_{\psi^{\theta}}(t-\tau)\nu_{\psi^{\theta}}(\tau)(f(\tau)-f(t))}dS_{\tau}, \quad t\in  \Gamma,
		\end{equation}
with
		\begin{equation}
		\Gamma_{t,r}:=\{\xi \in \Gamma: \, \abs{t-\xi}\leq r\}.
		\end{equation}
	\end{defi}	
In a similar way, the operators $[f]{^{\psi^{\theta}}}K_{\Gamma}$ and $[f]{^{\psi^{\theta}}}\mathscr{S}_{\Gamma}$ are defined  (the quaternionic Cauchy kernel appears on the right side of the integral and the normal vector is placed between the function $f$ and the kernel $\mathscr{K}_{\psi^{\theta}}$. It is worth pointing out that meanwhile ${^{\psi^{\theta}}}K_{\Gamma}[f]$ is left-$\psi^{\theta}$-hyperholomorphic, $[f]{^{\psi^{\theta}}}K_{\Gamma}$ is right-$\psi^{\theta}$-hyperholomorphic.
	
If $f$ $\in$ $C^{1}(\Omega, \, \mathbb{H(\mathbb{C})})\cap C(\Omega\cup\Gamma,\, \mathbb{H(\mathbb{C})})$, the Borel Pompieu formula is true
	\begin{equation}\label{fbpc}
	{^{\psi^{\theta}}}K_{\Gamma}[f](x)+{^{\psi^{\theta}}}T\big[{^{\psi^{\theta}}}D[f]\big](x)=\left\lbrace
	\begin{array}{ll}
	f(x) & \text{if} \, x\in\Omega_{+},\\
	0 &  \text{if} \, x\in\Omega_{-}.
	\end{array}
	\right..
	\end{equation}
Under the above assumptions, if moreover $f$ is left-$\psi^{\theta}$-hyperholomorphic in $\Omega$, from \eqref{fbpc} we get the Cauchy integral formula  
	\begin{equation} \label{ficc}
	{^{\psi^{\theta}}}K_{\Gamma}[f](x)=f(x), \quad x\in \Omega.
	\end{equation}
\begin{theorem}
Let $f\in C^{1}(\Omega,\, \mathbb{H(\mathbb{C})})\cap C(\Omega\cup\Gamma,\, \mathbb{H(\mathbb{C})})$. Then
	\begin{equation}
	{{^{\psi^{\theta}}}D}[{{^{\psi^{\theta}}}T}[f]](x):=\left\lbrace
	\begin{array}{ll}
	f(x) & \text{if} \, x\in\Omega_{+},\\
	0 &  \text{if} \, x\in \Omega_{-}.
	\end{array}
	\right.
	\end{equation}
	\end{theorem}
	\begin{theorem} \label{theodoresco}
		Let $f\in L_{p}(\Omega,\,\mathbb{H(\mathbb{R})})$, for $p> 3$. Then ${^{\psi^{\theta}}}T[f]$ satisfy the following inequalities
			\begin{equation} \label{des1t}
			\abs{{^{\psi^{\theta}}}T[f](x)}\leq C_{1}(\Omega,\, p)\abs{\abs{f}}_{L_{p}},
			\end{equation}
			\begin{equation}
			\abs{{^{\psi^{\theta}}}T[f](x)-{^{\psi^{\theta}}}T[f](x^{\prime})}\leq C_{2}(\Omega,\, p)\abs{\abs{f}}_{L_{p}}\abs{x-x^{\prime}}^{\frac{p-3}{p}}, \quad x\neq x^{\prime},
			\end{equation}
where			
			\begin{equation}
			\abs{\abs{f}}_{L_{p}}:=\bigg(\int_{\Omega} \abs{f}^{p}(\xi) dm(\xi)\bigg)^{\frac{1}{p}}.
			\end{equation}
		\end{theorem}
		\begin{proof} Applying a analogous reasoning  to that used in \cite[Theorem 2.3.2]{10} and the fact that $\abs{(x)_{\psi^{\theta}}}=\abs{x}$, the result is obtained.
		\end{proof}
	\begin{theorem} \label{tcd}
		Let $f\in C(\Gamma,\, \mathbb{H(\mathbb{C})})$ and suppose that
		\begin{equation} \label{condicion}
		\Psi_{\Gamma}[f](t):=\frac{1}{4\pi}\mathop{\lim}_{\delta \to 0} \int_{\Gamma\setminus \Gamma_{t,\delta}} \frac{{\abs{f(\xi)-f(t)}}_{c}}{\abs{\xi-t}^{2}} dS_{\xi},
		\end{equation}
exists uniformly with respect to $t\in\Gamma$. Then there exist the singular integral $^{\psi^{\theta}}\Phi[f]$ and the Cauchy transform $^{\psi^{\theta}}K[f]$ satisfies the Sokhotski-Plemelj formulas
		\begin{equation} \label{s-p}
		^{\psi^{\theta}}K_{\Gamma}^{\pm}[f](t):=\mathop{\lim}_{\Omega_{\pm} \ni x \to t}	{^{\psi^{\theta}}}K_{\Gamma}[f](x) =\frac{1}{2}({^{\psi^{\theta}}}\mathscr{S}_{\Gamma}[f](t) \pm f(t)), \, t\in\Gamma.
		\end{equation}
		\begin{proof}
Reasoning as in \cite[Theorem 3.1]{2} and applying that $\abs{(x)_{\psi^{\theta}}}=\abs{x}$ for all $x\in\mathbb{R}^{3}$ we get the proof.
		\end{proof}
	\end{theorem} 
\section{The Cauchy integral formula: vector fields case}
In this section we analyze the quaternionic Cauchy transform and the Cauchy integral formula restricted to the vector fields case.
	
Let $\mathbf{f}$ be a vector field  defined in $\Gamma$, then the Cauchy transform of $\mathbf{f}$ is given by:
	\begin{equation}
	\begin{split}
	{^{\psi^{\theta}}}K_{\Gamma}[\mathbf{f}](x)&
	=\int_{\Gamma}{\mathscr{K}_{\psi^{\theta}}(x-\xi)\langle\nu_{\psi^{\theta}}(\xi), \mathbf{f}(\xi)\rangle}dS_{\xi}+\int_{\Gamma}{\langle \mathscr{K}_{\psi^{\theta}}(x-\xi), [\nu_{\psi^{\theta}}(\xi), \mathbf{f}(\xi)]\rangle}dS_{\xi}+ \\ & \quad -\int_{\Gamma}{[\mathscr{K}_{\psi^{\theta}}(x-\xi),[\nu_{\psi^{\theta}}(\xi), \mathbf{f}(\xi)]]}dS_{\xi}.
	\end{split}
	\end{equation}
We can rewrite ${^{\psi^{\theta}}}K_{\Gamma}[\mathbf{f}]$ as
	\begin{equation}
	{^{\psi^{\theta}}}K_{\Gamma}[\mathbf{f}](x)=Sc({^{\psi^{\theta}}}K[\mathbf{f}](x))+Vec({^{\psi^{\theta}}}K[\mathbf{f}](x)),
	\end{equation}
where
	\begin{equation}
	Sc({^{\psi^{\theta}}}K_{\Gamma}[\mathbf{f}](x))=\int_{\Gamma}{\langle\mathscr{K}_{ \psi^{\theta}}(x-\xi),[\nu_{\psi^{\theta}}(\xi), \mathbf{f}(\xi)]\rangle}dS_{\xi},  
	\end{equation}
	\begin{equation}
	\begin{split}
	Vec({^{\psi^{\theta}}}K_{\Gamma}[\mathbf{f}](x))&=\int_{\Gamma}{\mathscr{K}_{\psi^{\theta}}(x-\xi)\langle\nu_{\psi^{\theta}}(\xi), \mathbf{f}(\xi)\rangle}dS_{\xi}+\\ & \quad -\int_{\Gamma}{[\mathscr{K}_{\psi^{\theta}}(x-\xi),[\nu_{\psi^{\theta}}(\xi), \mathbf{f}(\xi)]]}dS_{\xi}.
	\end{split}
	\end{equation}
From the previous observation we can see that in general for a vector field  $\mathbf{f}$  the Cauchy transform is not a purely vectorial complex quaternion. 
	\begin{remark}
If $\mathbf{f}\in\mathscr{M}_{\psi^{\theta}}$ the Cauchy transform is expressed as follows:
		\begin{equation}
		{^{\psi^{\theta}}}K_{\Gamma}[\mathbf{f}](x)=\int_{\Gamma}{\mathscr{K}_{\psi^{\theta}}(x-\xi)\langle\nu_{\psi^{\theta}}(\xi), \mathbf{f}(\xi)\rangle}dS_{\xi}  -\int_{\Gamma}{[\mathscr{K}_{\psi^{\theta}}(x-\xi),[\nu_{\psi^{\theta}}(\xi), \mathbf{f}(\xi)]]}dS_{\xi}.
		\end{equation}
	\end{remark}
In the next corollary the Cauchy integral formula is presented for the case of left-$\psi^{\theta}$-hyperholomorphic vector fields.
	\begin{corollary}
		Let $\mathbf{f}$ be a $\psi^{\theta}$-Laplacian vector field in $\Omega\cup\Gamma$ . Then
		\begin{equation} 
		\int_{\Gamma}{\mathscr{K}_{\psi^{\theta}}(x-\xi)\langle\nu_{\psi^{\theta}}(\xi), \mathbf{f}(\xi)\rangle}dS_{\xi}  -\int_{\Gamma}{[\mathscr{K}_{\psi^{\theta}}(x-\xi),[\nu_{\psi^{\theta}}(\xi), \mathbf{f}(\xi)]]}dS_{\xi} =\mathbf{f}(x),
		\end{equation}
		$x\in\Omega$.
		\begin{proof}
The proof of this theorem is obtained from the Cauchy integral formula.  \eqref{ficc}.
		\end{proof}
	\end{corollary}
	\begin{theorem} \label{thm 5} Let $\mathbf{f}\in \mathscr{M}_{\psi^{\theta}}$ be a continuous vector field and let the integral
		\begin{equation}\label{c1}
		\Psi_{\Gamma}[\mathbf{f}](t):=\frac{1}{4\pi}\mathop{\lim}_{\delta \to 0} \int_{\Gamma\setminus \Gamma_{t,\delta}} \frac{{\abs{\mathbf{f}(\xi)-\mathbf{f}(t)}}_{c}}{\abs{\xi-t}^{2}} dS_{\xi}.
		\end{equation}
exists uniformly with respect to $t\in\Gamma$. Then, for every  $t\in\Gamma$ there exists the Cauchy singular integral ${^{\psi^{\theta}}}\pmb{\mathscr{S}}_{\Gamma}[\mathbf{f}]$, the Cauchy transform ${^{\psi^{\theta}}}\pmb{K}_{\Gamma}[\mathbf{f}]$ has continuous limit values on $\Gamma$ and the following analogues of Sokhotski-Plemelj formulas hold:
		\begin{equation}
		{^{\psi^{\theta}}}\pmb{K}_{\Gamma}^{\pm}[\mathbf{f}](t):= \mathop{\lim}_{\Omega_{\pm} \ni x \to t}	{^{\psi^{\theta}}}\pmb{K}_{\Gamma}[\mathbf{f}](x) =\frac{1}{2}({^{\psi^{\theta}}}\pmb{\mathscr{S}}_{\Gamma}[\mathbf{f}](t) \pm \mathbf{f}(t)), \quad t\in\Gamma.
		\end{equation}
The class of continuous vector fields that satisfy  \eqref{c1} is denoted by $\mathscr{D}(\Gamma,\, \mathbb{H(\mathbb{C})})$.
		\begin{proof}
Let $\mathbf{f} \in \mathscr{M}_{\psi^{\theta}}\cap\mathscr{D}(\Gamma,\, \mathbb{H(\mathbb{C})}) $,  then the Cauchy transform  ${^{\psi^{\theta}}}\mathbf{K}_{\Gamma}[\mathbf{f}]:={^{\psi^{\theta}}}K_{\Gamma}[\mathbf{f}]$ is a complex quaternionic function with null scalar part. Now, we only apply the Theorem  \ref{tcd} to the function $f:=\mathbf{f}$ to get the result.
		\end{proof}
	\end{theorem}
\section{Boundary values of the Cauchy transform}
The following theorems provide two sets of assertions equivalent with the decomposition given by \eqref{des}.
	\begin{theorem} \label{thm 2} 
Let $\mathbf{f}\in\mathscr{D}(\Gamma,\, \mathbb{H(\mathbb{C})})\cap \mathscr{M}_{\psi^{\theta}}$ a vector field. Then the following statements are equivalent:
		\begin{itemize}
			\item [\textbf{1.}] The vector field $\mathbf{f}$ admits on $\Gamma$ a decomposition of the form \eqref{des}.
			\item [\textbf{2.}] The left-$\psi^{\theta}$-hyperholomorphic transform ${^{\psi^{\theta}}}K_{\Gamma}[\mathbf{f}]$ is right-$\psi^{\theta}$-hyperholomorphic in $\mathbb{R}^{3}\setminus\Gamma$.
			\item [\textbf{3.}] The vector field $Vec({^{\psi^{\theta}}}K_{\Gamma}[\mathbf{f}])$ is a $\psi^{\theta}$-Laplacian vector field.
			\item [\textbf{4.}] $Sc({^{\psi^{\theta}}}K_{\Gamma}[\mathbf{f}])=0$ in $\mathbb{R}^{3}$.
		\end{itemize}
		\begin{proof}  
We give only the proof of  $\textbf{3}$ $\Rightarrow$ $\textbf{4}$. The remaining implications can be proved using similar arguments as in \cite[Theorem 3.3]{2}. To this end, let us consider ${^{\psi^{\theta}}}K_{\Gamma}[\mathbf{f}]=Sc({^{\psi^{\theta}}}K_{\Gamma}[\mathbf{f}])+Vec({^{\psi^{\theta}}}K_{\Gamma}[\mathbf{f}])$. By hypothesis, $Vec({^{\psi^{\theta}}}K_{\Gamma}[\mathbf{f}])$ is a $\psi^{\theta}$-Laplacian vector field on $\mathbb{R}^{3}\setminus\Gamma$. Acting ${^{\psi^{\theta}}}D$ on $Sc({^{\psi^{\theta}}}K_{\Gamma}[\mathbf{f}])$ gives ${^{\psi^{\theta}}}\text{grad}\left[Sc({^{\psi^{\theta}}}K_{\Gamma}[\mathbf{f}])\right](x)=0$ for all $x\in\mathbb{R}\setminus\Gamma$.
			
For abbreviation, we write $L$  instead of $Sc(^{\psi^{\theta}}K_{\Gamma}[\mathbf{f}])$. It is a simple matter to show that
			\begin{equation}
			\notag
			\begin{split}
			&{^{\psi^{\theta}}}\text{grad}[L](x)=\frac{\partial L}{\partial x_{1}}(x)\textbf{i}+\frac{\partial L}{\partial x_{2}}(x)\textbf{i}e^{\textbf{i}\theta}\textbf{j}+\frac{\partial L}{\partial x_{3}}(x)e^{\textbf{i}\theta}\textbf{j}=\\ &=\frac{\partial L}{\partial x_{1}}(x)\textbf{i}+\left(\frac{\partial L}{\partial x_{3}}(x)\cos(\theta)-\frac{\partial L}{\partial x_{2}}(x)\sin(\theta)\right)\textbf{j}+\left(\frac{\partial}{\partial x_{3}}(x)\sin(\theta)+\frac{\partial L}{\partial x_{2}}(x)\cos(\theta)\right)\textbf{k}=0,
			\end{split}
			\end{equation}
is equivalent to the system
			\begin{equation}
			\notag
			\left\{
			\begin{array}{rcl}
			\displaystyle \frac{\partial L}{\partial x_{1}}(x) & = & 0, 
			\\ {}\\	\displaystyle \frac{\partial L}{\partial x_{3}}(x)\cos(\theta)-\frac{\partial L}{\partial x_{2}}(x)\sin(\theta) & = & 0,
			\\ {}\\	\displaystyle \frac{\partial L}{\partial x_{3}}(x)\sin(\theta)+\frac{\partial L}{\partial x_{2}}(x)\cos(\theta)  & = & 0. 
			\end{array}
			\right. 
			\end{equation}
whose solution is such that $\displaystyle\frac{\partial L}{\partial x_{s}}=0$ in $\mathbb{R}^{3}$, $s\in\{1, 2, 3\}$ . This implies that $Sc(^{\psi^{\theta}}K_{\Gamma}[\mathbf{f}])$ is constant in $\Omega_{\pm}$. Moreover, as $\mathbf{f}\in\mathscr{D}(\Gamma,\, \mathbb{H(\mathbb{C})})$, in view of the Sokhotski-Plemelj formulas 
			\begin{equation}
			\mathop{\lim}_{\Omega_{+} \ni x \to t}Sc(^{\psi^{\theta}}K_{\Gamma}[\mathbf{f}])(x)=\mathop{\lim}_{\Omega_{-} \ni x \to t}Sc(^{\psi^{\theta}}K_{\Gamma}[\mathbf{f}])(x)),
			\end{equation}
which implies that  $Sc(^{\psi^{\theta}}K_{\Gamma}[\mathbf{f}])$ is continuous in $\mathbb{R}^{3}$ and since $^{\psi^{\theta}}K_{\Gamma}[\mathbf{f}]$ vanish at infinity $\textbf{4}$ holds.
		\end{proof}
	
	\end{theorem}
	\begin{theorem} \label{thm 3} Let $\mathbf{f}\in\mathscr{D}(\Gamma,\, \mathbb{H(\mathbb{C})})$, then the following assertions are equivalent:
		\begin{itemize}
			\item [1.]  The Cauchy transform $^{\psi^{\theta}}K_{\Gamma}[\mathbf{f}]$ is right-$\psi^{\theta}$-hyperholomophic in $\mathbb{R}^{3}\setminus\Gamma$.
			\item [2.] $\mathbf{f}\in \mathscr{M}_{\psi^{\theta}}^{*}$.
			\item [3.] $^{\psi^{\theta}}\mathscr{S}_{\Gamma}[\mathbf{f}]=[\mathbf{f}]{^{\psi^{\theta}}}\mathscr{S}_{\Gamma}$.
		\end{itemize}
		\begin{proof}
			The proof is obtained reasoning as in \cite[Theorem 3.4]{2}.
		\end{proof}
	\end{theorem}
\section{Main result}
Before beginning the proof of our main result we give some important remarks.
	
Let $\mathbf{f}$ be a $\psi^{\theta}$-Laplacian vector field. Applying ${^{\psi^{\theta}}}D$ and $D^{\psi^{\theta}}$ to $\mathbf{f}$ we obtain that equations
	\begin{equation}
	\notag
	{^{\psi^{\theta}}}D[\mathbf{f}]=-{^{\psi^{\theta}}}\text{div}[\mathbf{f}]+{^{\psi^{\theta}}}\text{rot}[\mathbf{f}]=0,
	\end{equation}
	\begin{equation}
	\notag
	D^{\psi^{\theta}}[\mathbf{f}]=-{^{\overline{\psi^{\theta}}}}\text{div}[\mathbf{f}]+{^{\overline{\psi^{\theta}}}}\text{rot}[\mathbf{f}]=0,
	\end{equation}
are equivalent. In other words, the class of $\psi^{\theta}$-Laplacian vector fields is identified with the set of purely vectorial  $\psi^{\theta}$-hyperholomorphic functions. 
	
By definition of $\mathscr{M}_{\psi^{\theta}}$ and $\mathscr{M}_{\psi^{\theta}}^{*}$, we have	
	\begin{equation}
	\notag
	\mathscr{D}(\Gamma, \mathbb{H(\mathbb{C})})\cap \mathscr{M}_{\psi^{\theta}}:=\bigg\{\mathbf{f} \in \mathscr{D}(\Gamma,\, \mathbb{H(\mathbb{C})}): \int_{\Gamma}{\langle\mathscr{K}_{\psi^{\theta}}(x-\xi),[\nu_{\psi^{\theta}}(\xi), \mathbf{f}(\xi)]\rangle}dS_{\xi}=0,  x\not\in\Gamma \bigg\},
	\end{equation}
	\begin{equation}
	\notag
	\mathscr{D}(\Gamma, \mathbb{H(\mathbb{C})})\cap\mathscr{M}_{\psi^{\theta}}^{*}:=\bigg\{\mathbf{f} \in \mathscr{D}(\Gamma,\, \mathbb{H(\mathbb{C})}): \int_{\Gamma}{\langle\mathscr{K}_{\psi^{\theta}}(x-\xi),[\nu_{\psi^{\theta}}(\xi), \mathbf{f}(\xi)]\rangle}dS_{\xi}=0,  x\in\Gamma \bigg\}.
	\end{equation}
Theorem \ref{thm 5} now shows that
	\begin{equation}
	\notag
	2\lim_{\Omega_{\pm} \ni x \to t}Sc({^{\psi^{\theta}}}K_{\Gamma}[\mathbf{f}])(x)=2Sc({^{\psi^{\theta}}}K^{\pm}_{\Gamma}[\mathbf{f}])(t)=2Sc({^{\psi^{\theta}}}\mathscr{S}_{\Gamma}[\mathbf{f}])(t),
	\end{equation}
which implies that
	\begin{equation}
	\notag
	\text{if}\, \mathbf{f}\in \mathscr{D}(\Gamma,\, \mathbb{H(\mathbb{C})})\cap \mathscr{M}_{\psi^{\theta}}, \,\, \text{then}\,\, \mathbf{f} \in \mathscr{D}(\Gamma,\, \mathbb{H(\mathbb{C})})\cap\mathscr{M}_{\psi^{\theta}}^{*}.
	\end{equation}	
	
Now we are in conditions to state and proof our main result. 

	\begin{theorem} \label{thm 6} Let $\Omega$ be a Jordan domain in  $\mathbb{R}^{3}$  with rectifiable boundary $\Gamma$. Let $\mathbf{f}\in Lip_{\mu}(\Gamma,\, \mathbb{H(\mathbb{C})})\cap\mathscr{M}_{\psi^{\theta}}$. Then,  for $\displaystyle \frac{2}{3}<\mu \leq 1$ the vector field $\mathbf{f}$ admits on $\Gamma$ a unique decomposition of the form \eqref{des}.
\begin{proof} The proof will be divided into two steps.

\medskip
\noindent			
\textbf{Existence:}	
Let $\mathbf{f}\in Lip_{\mu}(\Gamma,\, \mathbb{H(\mathbb{C})})\cap\mathscr{M}_{\psi^{\theta}}$. Inspired by the Borel Pompieu formula, we consider a new kind of Cauchy transform given by:
			\begin{equation}\label{tc}
			\begin{split}
			{^{\psi^{\theta}}}K^{*}_{\Gamma}[\mathbf{f}](x):=-{^{\psi^{\theta}}}T[{^{\psi^{\theta}}}D[{\bf f}^{w}]](x)+ {\bf f}^{w}(x), \quad x\in \mathbb{R}^{3}\setminus\Gamma.
			\end{split}
			\end{equation}

We can write ${\bf f}^{w}={\bf f}_{1}^{w}+i {\bf f}_{2}^{w}$, by Proposition \ref{lemmalp} it follows that ${^{\psi^{\theta}}}D[{\bf f}_{1,2}^{w}]\in L_{p}(\Omega,\, \mathbb{H})$ for $p<\displaystyle\frac{1}{1-\mu}$. The condition $\displaystyle\frac{2}{3}<\mu$ implies that ${^{\psi^{\theta}}}D[{\bf f}_{1,2}^{w}]\in L_{p}(\Omega,\, \mathbb{H})$ for some $p>3$ . Due the fact that 
\begin{equation}\label{tc2}
-{^{\psi^{\theta}}}T[{^{\psi^{\theta}}}D[{\bf f}^{w}]](x)=-{^{\psi^{\theta}}}T[{^{\psi^{\theta}}}D[{\bf f}^{w}_{1}]](x)-i{^{\psi^{\theta}}}T[{^{\psi^{\theta}}}D[{\bf f}^{w}_{2}]](x), \quad x\in \mathbb{R}^{3}\setminus\Gamma.
\end{equation}
 by Theorem \ref{theodoresco} the integrals on the right side of \eqref{tc2} exist and their sum represent a continuous function in $\mathbb{R}^{3}$. Hence, ${^{\psi^{\theta}}}K^{*}_{\Gamma}[\mathbf{f}]$ exists and possesses continuous extensions to the closures of the domains $\Omega_{\pm}$.  Consequently, the problem (\ref{des}) has a solution given by
			\begin{equation}
			{\bf F}^{+}(x):=-{^{\psi^{\theta}}}T[{^{\psi^{\theta}}}D[{\bf f}^{w}]](x)+{\bf f}^{w}(x),
			\end{equation}
			\begin{equation}
			{\bf F}^{-}(x):={^{\psi^{\theta}}}T[{^{\psi^{\theta}}}D[{\bf f}^{w}]](x).
			\end{equation}	
			
The difference between the limit values of ${^{\psi^{\theta}}}K^{*}_{\Gamma}[\mathbf{f}](x)$ in $\Gamma$ approaching $x$ from $\Omega_{\pm}$ is equal to
			\begin{equation}
			{\bf F}^{+}(t)+{\bf F}^{-}(t)={\bf f}^{w}(t)=\mathbf{f}(t), \quad \forall \, t\in\Gamma.
			\end{equation}	
			
The vector fields ${\bf F}^{\pm}$ are left-$\psi^{\theta}$-hyperholomorphics in $\Omega_{\pm}$ and each has null scalar part, therefore ${\bf F}^{\pm}$ are right-$\psi^{\theta}$-hyperholomorphics in $\Omega_{\pm}$. 
\medskip

\noindent			
\textbf{Uniqueness:}
Since ${\bf F}^{\pm}={\bf F}_{1}^{\pm}+i{\bf F}_{2}^{\pm}$, it follows that $\mathbf{f}=\mathbf{f}_{1}+i\mathbf{f}_{2}$, where
			\begin{equation}
			\mathbf{f}_{1}(t)={\bf F}_{1}^{+}(t)+{\bf F}_{1}^{-}(t), \quad \forall \, t\in\Gamma,
			\end{equation}
			\begin{equation}
			\mathbf{f}_{2}(t)={\bf F}_{2}^{+}(t)+{\bf F}_{2}^{-}(t), \quad \forall \, t\in\Gamma, 
			\end{equation}
and ${\bf F}_{1,2}^{\pm}$ are left-right-$\psi^{\theta}$-hyperholomorphic in $\Omega_{\pm}$, hence it is enough to prove that ${\bf F}_{1,2}^{\pm}$ are unique. Indeed, the uniqueness of such functions follow from a sort of Painlev\'e Theorem (see \cite[Corollary 5.3]{18}) taking into account that the considered structural set plays no role in the proof. 
		\end{proof}
	\end{theorem}

\section*{Acknowledgements}
This is part of the DGC's Master thesis, written under the supervision of JBR and MAPR at the Escuela Superior de F\'isica y Matem\'aticas.  Instituto Polit\'ecnico Nacional. DGC gratefully acknowledges the financial support of the Postgraduate Study Fellowship of the Consejo Nacional de Ciencia y Tecnolog\'ia (CONACYT).	JBR and MAPR were partially supported by Instituto Polit\'ecnico Nacional in the framework of SIP programs and by Universidad de las Am\'ericas Puebla, respectively.

\end{document}